\documentclass{amsart}
\usepackage{mathtools} \mathtoolsset{showonlyrefs}
\usepackage{enumitem}
\usepackage{verbatim}

\newtheorem{theorem}{Theorem}[section]

\newtheorem{corollary}[theorem]{Corollary}
\newtheorem{definition}[theorem]{Definition}
\newtheorem{remark}[theorem]{Remark}
\newtheorem{proposition}[theorem]{Proposition}
\def \R{\mathbb R}
\def \N{\mathbb N}
\newcommand {\B}[1][n] {B_2^{#1}}
\def \co{\operatorname{conv}}
\def \s{S^{n-1}}
\def \I{\operatorname{Id}}
\renewcommand {\S}{S^{nm-1}}

\renewcommand{\chi}{\operatorname{1}}
\newcommand{\vol}[2][n]{\left|#2 \right|_{#1}}
\newcommand{\Vol}[1]{\vol[nm]{#1}}

\newcommand {\M}[1][n,m] {\operatorname{M}_{#1}(\R)}
\newcommand {\conbod}[1][n,m] {\mathcal{K}^{#1}}

\title[A Rogers--Brascamp--Lieb--Luttinger inequality]{A Rogers--Brascamp--Lieb--Luttinger inequality in the space of matrices}
\author{J. Haddad}
\address{Departamento de An\'alisis Matem\'atico, Facultad de Matem\'aticas, Universidad de Sevilla, Sevilla, Spain}
\email{jhaddad@us.es}

\begin{document}

\begin{abstract}
	We consider convex bodies in $\M$, the space of matrices of $n$-rows and $m$-columns.
	A special case of fiber-symmetrization in $\M$ was recently introduced in \cite{HLPRY23,HLPRY23_2}.
	We prove a Rogers--Brascamp--Lieb--Luttinger type inequality with respect to this symmetrization and provide some applications.
\end{abstract}

\maketitle

\section{Introduction}
Let $K \subseteq \R^n$ be a Lebesgue-measurable set, and $v$ be a unit vector in $\R^n$.
The Steiner symmetrical of $K$ with respect to $v$ is
\[S_vK = \{x + t v \in \R^n: t\in \R, x \in P_{v^\perp} K, |t| \leq \ell_v(x)/2\}\]
where $P_{v^\perp}K$ is the orthogonal projection of $K$ on the hyperplane orthogonal to $v$ and $\ell_v(x)$ is the one-dimensional measure of the set obtained intersecting $K$ with the line parallel to $v$ passing through $x$.

Steiner symmetrization has a huge number of applications in geometry ranging from isoperimetric inequalities to functional analysis and it is at the core of the solution of many problems in Convex Geometry.
For an overview of the many applications we refer to the books \cite{AGA, Sh1} and the references therein.
The operator $S_v$ enjoys two important properties. First, it preserves the Lebesgue measure, and second, any convex body (compact, convex with non-empty interior) can be transformed by an infinite sequence of successive symmetrizations, into a Euclidean ball.

Let $f:\R^n \to \R$ be a non-negative measurable function and $t\geq 0$. Denote $\{f \geq t\} := \{x \in \R^n : f(x) \geq t\}$, and similarly with $\leq$.
We say that $f$ is quasi-concave if its super-level sets $\{f \geq t\}$ are convex for every $t \in \R$.
A function $f$ is quasi-convex if its sub-level sets $\{f \leq t\}$ are convex for every $t \in \R$.

The Steiner symmetrization of a quasi-concave function $f$ with respect to $v$ is defined by
\[f^{(v)}(x) = \int_0^\infty \chi_{S_v\{f\geq t\}}(x) dt\]
where $\chi_K$ is the characteristic function of the set $K$.
The function $f^{(v)}$ can be characterized by being the unique function whose super-level sets are the Steiner symmetricals of the super-level sets of $f$, in the direction $v$.

Rogers \cite{rogers1957single} and Brascamp, Lieb and Luttinger \cite{brascamp1974general} proved independently a general integral inequality, which generalizes the celebrated Riesz convolution inequality:

\begin{theorem}
	\label{thm_bll_original}
	Let $k,d \geq 1$ and $f_i:\R^n \to \R$ with $1 \leq i \leq k$, be non-negative and measurable functions, and let $a_j^{(i)}$ be real numbers with $1 \leq j \leq d, 1 \leq i \leq k$, then
	\begin{multline}
		\label{eq_bll_original}
		\int_{\R^n} \cdots \int_{\R^n} \prod_{i=1}^k  f_i \left( \sum_{j=1}^d a_j^{(i)} x_j \right) d x_1 \ldots d x_d \\
		\leq \int_{\R^n} \cdots \int_{\R^n} \prod_{i=1}^k f^{(v)}_i \left(\sum_{j=1}^d a_j^{(i)} x_j \right) d x_1 \ldots d x_d.
	\end{multline}
\end{theorem}
Notice that if the matrix with entries $a_j^{(i)}$ has rank less than $d$, the integrals at both sides are either $0$ (if all $f_i$ are zero a.e.) or infinite.
If $k=d$ and the matrix is invertible, there is equality by a change of variables, and the fact that symmetrization preserves the integral of each function.
This inequality can be applied iteratively for different directions $v \in S^{n-1}$, and by taking limits, we may replace $f^{(v)}$ in the right-hand side of \eqref{eq_bll_original} by the symmetric decreasing rearrangement $f^*$ (see \cite{brascamp1974general}).

An interesting generalization of Theorem \ref{thm_bll_original} was used to study randomized isoperimetric inequalities by Paouris and Pivovarov \cite{PP17-2} (Theorem \ref{thm_bll_paouris_pivovarov} below). This inequality is attributed to Christ \cite{christ1984estimates}.  The following definition is given in \cite{PP17-2}.
\begin{definition}
	\label{def_steinerconcave}
	A function $f:(\R^n)^d \to \R$ with $d \geq 1$ is called Steiner-concave (resp.\ Steiner-convex) if for every $v \in S^{n-1}$ and $y_1, \ldots, y_d \in v^\perp$, the function $G:\R^d \to \R$ defined by
	\[G(t_1, \ldots, t_d) = f(y_1 + t_1 v, \ldots, y_d + t_d v)\]
	is even and quasi-concave (resp.\ quasi-convex).
\end{definition}
The evenness in this definition is of absolute central importance, as we shall see in the course of the paper.

\begin{theorem}[{\cite[Theorem 3.8]{PP17-2}}]
	\label{thm_bll_paouris_pivovarov}
	Let $f_1, \ldots, f_{k_1}$ with $k_1 \geq 1$, be non-negative integrable functions on $\R^n$ and let $a_j^{(i)}$ be real numbers with $1 \leq j\leq d, 1 \leq i \leq k_1$.
	Let $F^{(i)}: (\R^n)^d \to \R$ with $d\geq 1$ and $1 \leq i \leq k_2$ be non-negative Steiner-concave functions, and let $\mu$ be a non-negative measure with a rotationally invariant quasi-concave density in $\R^n$. Then,
	\begin{multline}
		\int_{\R^n} \cdots \int_{\R^n}  \prod_{i=1}^{k_2} F^{(i)}(x_1, \ldots, x_d) \prod_{i=1}^{k_1} f_i\left(\sum_{j=1}^d a_j^{(i)} x_j\right) d\mu(x_1) \ldots d\mu(x_d)\\
		\leq 
		\int_{\R^n} \cdots \int_{\R^n} \prod_{i=1}^{k_2} F^{(i)}(x_1, \ldots, x_d) \prod_{i=1}^{k_1} f^{(v)}_i\left(\sum_{j=1}^d a_j^{(i)} x_j\right) d\mu(x_1) \ldots d\mu(x_d).\\
	\end{multline}
\end{theorem}

We shall see that this theorem can be generalized in an elegant way if one considers a symmetrization procedure adapted to the product structure of $(\R^n)^d$.
Both sets of functions $F^{(i)}, f_i$ together with the measure $\mu$ will be seen as objects of the same nature under this perspective.
Here each $F^{(i)}$ is a function that is already symmetric in a sense that will be made clear next.

Throughout the paper $d,n,m \geq 1$ will be fixed natural numbers.
It is natural to identify $(\R^n)^m$ with the set of matrices of $n$ rows and $m$ columns, $\M$.
For $(x_1, \ldots, x_m) \in (\R^n)^m$ we write $x \in \M$ for the matrix whose columns are the vectors $x_i$.
The expression $\sum_{j=1}^d a_j^{(i)} x_j$ in inequality \eqref{eq_bll_original} can be written as a matrix product $x a^{(i)}$, where $a^{(i)}$ is the column vector whose entries are $a_j^{(i)}$.

In \cite{McM99} McMullen defined the fiber combination of convex sets, and later in \cite{BGG17} Bianchi, Gardner and Gronchi defined the fiber-symmetrization with respect to a subspace of $\R^d$.
We will use this symmetrization procedure in $\M$, with respect to a specific set of subspaces.
This particular case of the fiber symmetrization, which is well adapted to the product structure of $\M = (\R^n)^m$, has found applications only recently, see \cite{HLPRY23} and \cite{HLPRY23_2}.

\begin{definition}
	\label{def_subspaces}
	For $v \in\s \subseteq \M[n,1]$, consider the $m$-dimensional space 
	\[v^m := \{v t: t \in \M[1,m]\} \subseteq \M[n,m]\]
	and let 
	\[v^{\perp m} = \{x \in \M: v^t x = 0 \in \M[1,m]\}.\]
\end{definition}
The subspace $v^m$ consists of matrices whose columns are all multiples of $v$, while $v^{\perp m}$ consists of matrices, whose columns are all perpendicular to $v$.
Of course, we have $\M = v^m \oplus v^{\perp m}$.

For any set $K \subseteq \R^m$, its difference body is the set $D K = \{x-y \in \R^m: x,y \in K\}$.

\begin{definition}
	\label{def_setsym}
	Let $K \subseteq \M$ be a measurable set and $v \in\s$.
	We define the $m$-th higher-order Steiner symmetrical of $K$ with respect to $v$ by
	\[
		\bar S_v K = \bigcup_{y \in v^{\perp m}} \left( y + \frac 12 D(K \cap (y + v^m)) \right),
	\]
Furthermore, we have
	\begin{multline}
		\label{eq_sym_def}
		\bar S_v K=
		\Big\{y+ v \frac{t-s}2 \in \M \colon y \in v^{\perp m}, t,s \in \M[1,m],\\ \text{ and } y + v t, y + v s \in K \Big\}.
	\end{multline}
\end{definition}

For convenience, we adopt the following definition:
\begin{definition}
	\label{def_quasiconcave}
	Let $f:\M \to \R$ be a measurable function.
	We say that $f$ is $m$-quasi-concave (resp.\ $m$-quasi-convex) if for every $v \in S^{n-1}$, $y \in v^{\perp m}$, the function $t \in \M[1,m] \mapsto f(y+vt)$ is quasi-concave (resp.\ quasi-convex).
\end{definition}
A $1$-quasi-concave (resp.\ $1$-quasi-convex) function is just a quasi-concave (resp.\ quasi-convex) function $f:\R^n \to \R$.
If a function $f:\M \to \R$ is quasi-concave (resp.\ convex) in the usual sense, then it is $m$-quasi-concave (resp.\ $m$-quasi-convex), since one requires the concavity (resp.\ convexity) of the restriction of $f$ to a specific set of $m$-dimensional subspaces, instead of the whole space.

The function $x \in \M[n,n] \mapsto |\det(x)|$ is $n$-quasi-convex but not convex. To see this, use the multilinearity of the determinant to write
\[|\det(x+vt)| = \left|\det(x) + \sum_{i=1}^n t_i \det(\tilde x_i) \right|\]
where $x_i$ is the matrix obtained replacing the $i$-th column of $x$ by $v$, and $t_i$ are the coordinates of $t$.

Notice that the notion of Steiner-concavity of an $m$-quasi-concave function corresponds to the fact that its super-level sets are invariant by $\bar S_v$ for every $v \in \s$.
In particular, the characteristic function $\chi_K$ of a convex set $K \subseteq \M$, is Steiner-concave if and only if $\bar S_v K = K$ for every $v \in S^{n-1}$.

More examples of Steiner-concave functions (which are in particular $m$-quasi-concave), can be found in \cite[Corollary 4.2, Corollary 4.5, Theorem 4.7, Theorem 4.8 and Proposition 4.9]{PP17-2}.
\begin{remark}
	\label{rem_not_only_convex}
	Notice that for $m=1$, $v^m$ is just a line and, if $K$ is convex, the intersection $K \cap (y + v^m)$ is an interval, while $y + \frac 12 D(K \cap (y + v^m))$ is the translated interval centered at $y$. This is, for $m=1$ and $K$ convex, the operator $\bar S_v$ is the usual Steiner symmetrization operator $S_v$.
	However, if $K$ is not convex this is no longer true, as the difference body of an arbitrary set might not be an interval.
\end{remark}

\begin{definition}
	\label{def_funcsym}
	Let $f:\M \to \R$ be non-negative and $m$-quasi-concave.
	Define the (higher-order) Steiner symmetrical of $f$ as
	\[f^{(v)}(x) = \int_0^\infty \chi_{\bar S_v \{f \geq t\}}(x) d t.\]
\end{definition}

\begin{definition}
	\label{def_funcsym_convex}
	If $f:\M \to \R$ is a non-negative $m$-quasi-convex function we define
	\[f_{(v)}(x) = \int_0^\infty (1 - \chi_{\bar S_v \{f < t\}}(x)) d t.\]
\end{definition}
Clearly $f^{(v)}$ and $f_{(v)}$ are also $m$-quasi-concave and $m$-quasi-convex, respectively.

The operator $\bar S_v$ was recently used in \cite{HLPRY23,HLPRY23_2} to prove the higher-order Petty projection inequality.
In this note we study the operator $\bar S_v$ acting on sets and functions, and prove a Rogers--Brascamp--Lieb--Luttinger type inequality.
Some applications will be given.
Our main theorem reads as follows:
\begin{theorem}
	\label{thm_bll_with_linears}
	Let $k \geq 1$ and $f_i:\M \to \R, 1 \leq i \leq k$, be non-negative and $m$-quasi-concave, and let $L_i \in \M[d,m]$ with $1\leq i \leq k$ and $d \geq m$ be real matrices of rank $m$, then
	for $v \in S^{n-1}$,
	\[ \int_{\M[n,d]} \prod_{i=1}^k f_i (x L_i) d x \leq \int_{\M[n,d]} \prod_{i=1}^k f_i^{(v)} (x L_i) d x.\]
\end{theorem}

Inequality \eqref{eq_bll_original} in the quasi-concave case corresponds to the case where $m=1$ so all matrices $L_i$ are non-zero column vectors.
Theorem \ref{thm_bll_paouris_pivovarov} is a particular case of Theorem \ref{thm_bll_with_linears} where $m=d$, $L_i = \I_d, i=1, \ldots, k_1$, and the first $k_1$ functions satisfy $f_i^{(v)} = f_i, \forall v \in \s$, while the rest of the functions ($f_i$ and $\frac{d\mu}{d x_i}$) depend on only one variable.

We insist here that Theorem \ref{thm_bll_with_linears} extends the aforementioned results, only for quasi-concave functions.
If the functions $f_i$ are general measurable functions, there is the possibility that $f_i^{(v)}$ is no longer measurable. This owes to the fact that the difference body of a Lebesgue measurable set may be non-measurable.
This problem can be overcome by using a measure-theoretic definition of $D$ and $\bar S_v$, but we shall not pursue this line of research here.
For general measurable non-negative functions the inequality still holds, since Theorem \ref{thm_bll_with_linears} is proved using only the Brunn--Minkowski inequality, which is valid also in this context. But we obtain a weaker result, due to the observation in Remark \ref{rem_not_only_convex}. For this reason, and for a clearer exposition, we deal here only with convex sets and $m$-quasi-concave or $m$-quasi-convex functions.

The rest of the paper is organized as follows: In Section \ref{sec_preliminaries} we recall basic facts about convexity
and prove basic properties related to the operator $\bar S_v$ and its analogue on functions.
In Section \ref{sec_inequalities} we establish symmetrization inequalities and prove Theorem \ref{thm_bll_with_linears}.
In Section \ref{sec_applications} we provide some applications and finally in Section \ref{sec_invariant_sets} we study the sets that remain invariant under $\bar S_v$.

\section{Notation and Preliminaries}
\label{sec_preliminaries}

\subsection{Linear Structure}
We consider the vector space $\M$ of real matrices of $n$ rows and $m$ columns, with the usual Euclidean structure given by
\begin{equation}
	\label{def_inner_product}
	\langle A,B \rangle = \sum_{i=1}^n \sum_{j=1}^m A_{i,j} B_{i,j},\ \ \ \|A\|_2 = \sqrt{\sum_{i=1}^n \sum_{j=1}^m A_{i,j}^2}
\end{equation}
for $A,B \in \M$. Here $A_{i, j}$ denote $(i, j)$-th entry of $A\in \M$.

The Lebesgue measure in $\M$ is inherited from the natural identification between $\M$ and $\R^{nm}$.
We denote $\vol[n]{\ \cdot\ }$ and $\vol[nm]{\ \cdot\ }$ for the volume in $\R^n$ and $\M$ respectively.

For notational convenience we identify $\R^n$ with $\M[n,1]$, and $\R^m$ with $\M[1,m]$, unless stated otherwise (we follow this convention throughout the paper, with the exception of Section \ref{subsec_operatornorms}).
We denote the Euclidean unit sphere $\s \subseteq \R^n = \M[n,1]$, and $\mathbb{S}^{nm-1} \subseteq \M$.
Also, we identify $\M[n,a+b] = \M[n,a] \times \M[n,b]$ in the natural way.

For a set $U \subseteq \M$, $A \in \M[k,n],$ and $B \in \M[m,l]$, we write
\[A U = \{A x \in \M[k,m] : x \in U\} \ \ \mathrm{and}\ \ \ U B = \{x B \in \M[n,l] : x \in U\}.\]

For any $n\geq 1$, $GL_n \subseteq \M[n,n]$ will denote the subset of invertible matrices.

\subsection{Basic Convexity}
For a detailed account of the theory of convex bodies, we refer to Schneider's book \cite{Sh1}.
We denote the set of convex bodies in $\M$ by $\conbod$.
The support function of $K \in \conbod$ is given by
\[h_K(x) = \sup\{ \langle x,y \rangle : y \in K \}\]
for $x \in \M$, where the inner product is the one given in \eqref{def_inner_product}.
If $K$ contains the origin in the interior, the polar body $K^\circ$ defined by
\[K^\circ = \{x \in \M: h_K(x) \leq 1 \},\]
is also a convex body.

\begin{definition}
	Let $A \in \M[d,m]$ then $A$ induces a linear map
	\[\bar A : \M[n,d] \to \M[n,m]\]
	by right-multiplication, $\bar A(x) = xA$.
\end{definition}
	Observe that 
	\begin{equation}
		\label{def_functoriality}
		\overline{A B} = \bar B \circ \bar A,\  \overline{I_m} = I_{nm}.
	\end{equation}

The general Brunn-Minkowski inequality states that if $K,L \subseteq \M$ are convex bodies, then
\begin{equation}
	\label{eq_brunn_minkowsky}
	\Vol{K+L}^{\frac 1{nm}} \geq \Vol{K}^{\frac 1{nm}} + \Vol{L}^{\frac 1{nm}}
\end{equation}
where $K+L = \{x+y:x\in K, y\in L\}$ is the Minkowski sum of sets.
Equality holds in \eqref{eq_brunn_minkowsky} if and only if $K$ and $L$ are homothetic.

A convex body $K \in \conbod$ is origin-symmetric if $K = -K$, while it is symmetric with respect to a point $c \in \M$ if $K-c$ is origin-symmetric.
The difference body of $K$ is defined by $D K = K + (-K)$.
One can see analyzing the support functions of $K$ and $\frac 12 DK$, that $K = \frac 12 D K$ if and only if $K$ is origin-symmetric. If $K$ is symmetric with respect to a point, then $\frac 12 D K$ is just a translate of $K$.
By the Brunn-Minkowski inequality, $\Vol{\frac 12 D K} \geq \Vol{K}$ with equality if and only if $K$ is symmetric with respect to some point.

The mean width of $K$ is $w(K) = 2 \int_{\S} h_K(x) d\sigma(x)$ where $\sigma$ is the rotation invariant probability measure on $\S$.

The Hausdorff distance between two convex bodies $K,L$ is defined by $\|h_K - h_L\|_\infty$, where $\|\cdot\|_\infty$ denotes the supremum norm on the sphere.
For any measurable function $f:S^{nm-1} \to \R$, the $L_p$ norm will be denoted by $\|f\|_p$.

For $v \in S^{n-1}$ we denote by $R_v \in \M[n,n]$ the matrix of the reflection in $\R^n$ with respect to $v^\perp$.
This is, $R_v(w + \lambda v) = w - \lambda v$ for every $w \in v^\perp$ and $\lambda \in \R$.
Notice that also $R_v(x+v t) = x-v t$ for every $x \in v^{\perp m}$ and $t \in \M[1,m]$.

\subsection{Symmetrization}

Here we establish basic properties of the fiber symmetrization in convex sets and $m$-quasi-concave functions.
Some of these results appeared already in \cite{HLPRY23,HLPRY23_2}.

	\begin{proposition}
		\label{prop_set_properties}
	The operator $\bar S_v$ satisfies the following properties:
		\begin{enumerate}[
					label=\ref{prop_set_properties}-\arabic*,
					ref=\ref{prop_set_properties}-\arabic*
				]
	\item
	\label{prp_symincreases_volume}
	If $K$ is a convex body,
		$\Vol{\bar S_v K} \geq \Vol{K}$ with equality if and only if $K \cap (y+v^m)$ is symmetric with respect to a point possibly depending on $y$, for almost every $y \in v^{\perp m}$.
	\item
	\label{prp_sym_is_steiner}
				If $K \subseteq \M[n,1]$ is convex, then $\bar S_v K = S_v K$ is the usual Steiner symmetrization. In particular it preserves $n$-dimensional volume.
\item
	\label{prp_sym_is_difference}
	If $K \subseteq \M[1,m]$ and $v \in S^0 = \{-1,+1\}$, then $\bar S_{\pm 1} K = \frac 12D K$.

\item
	\label{prp_sym_monotonicity}
	If $K \subseteq L$ then
	$\bar S_v K \subseteq \bar S_v L$.
\item
	\label{prp_sym_preserves_convexity}
	If $K$ is a convex body, then so is $\bar S_v K$.
\item
	\label{prp_sym_left_commutation}
	Let $A \in \M[n,n]$ be an orthogonal matrix and $K \in \conbod, v \in \s$, then 
	\[\bar S_{A v} (A K) = A \bar S_v K.\]
\item 
	\label{prp_sym_right_commutation}
	If $K \subseteq \M[n,d]$ is a convex set and $A \in \M[d,m]$ is any matrix, then we have
	\[(\bar S_v K) A \subseteq \bar S_v (K A) .\]
	Moreover, if the rank of $A$ is $d$ (we implicitly assume $m \geq d$), then there is equality.

\item
		\label{prp_sym_product}
	Assume $r\geq 1$ and we have convex bodies $K_i \in \conbod[n,m_i]$ with $1 \leq i \leq r$, $m_i \in \N$, $m_1+\cdots +m_r = m$, and $K_1 \times \cdots \times K_r \subseteq \M$, then
	\begin{equation}
		\bar S_v (K_1 \times \cdots \times K_r) = \bar S_v K_1 \times \cdots \times \bar S_v K_r.
	\end{equation}
\item
	\label{prp_sym_right_submersion}
	If $K \in \conbod[n,m]$ and $A \in \M[d,m]$ is a rank $m$ matrix (we implicitly assume $d \geq m$), then
	\[\bar A^{-1}( \bar S_v K ) = \bar S_v \bar A^{-1} (K)\]
	where $\bar A^{-1}$ denotes the preimage by the function $\bar A$.
	\end{enumerate}
\end{proposition}
\begin{proof}
	\ 

\medskip\ref{prp_symincreases_volume}
If $L \subseteq\R^m$ is a convex body, by the Brunn-Minkowski inequality, $\vol[m]{\frac 12D L} \geq \vol[m]{L}$ with equality if and only if $L$ is symmetric with respect to some point.
Then by Fubini, $\Vol{\bar S_v K} \geq \Vol{K}$ with equality if and only if $K \cap (y+v^m)$ is symmetric with respect to a point possibly depending on $y$, for almost every $y \in v^{\perp m}$.

\medskip\ref{prp_sym_is_steiner}
	If $m = 1$ and $K \subseteq \R^n$ is convex, then $K \cap (x+v^m)$ is a one dimensional interval, and $\frac 12D\left(K \cap (x+v^m)\right)$ is the same interval centered at the origin.

\medskip\ref{prp_sym_is_difference}
	This is clear from the definition.

\medskip\ref{prp_sym_monotonicity}
	This follows from the fact that $D$ is a monotone operator.

\medskip\ref{prp_sym_preserves_convexity}
	This fact was proven in \cite[Theorem 2.3]{McM99}.

\medskip\ref{prp_sym_left_commutation}
	This is clear from the definition.

\medskip\ref{prp_sym_right_commutation}
	Let $z \in (\bar S_v K)A$. By formula \eqref{eq_sym_def},
	\[z = (x+v\frac{t-s}2)A,\text{ with } x+v t, x+v s \in K, \text{ and } v^tx=0.\]
	Clearly $xA+vtA, xA+vsA \in K A$ and $v^t x A = 0$, which implies that
	\[z = xA+v\frac{tA-sA}2 \in \bar S_v(K A).\]

	Conversely (and assuming $A$ has rank $d$), let $z \in \bar S_v(K A)$, then
	\[z = x+v\frac{t-s}2\]
	with $t, s \in \M[1,m]$, $x \in \M[n,m]$, $x+v t, x+v s \in K A$ and $v^t x=0$.

	We can write
	\begin{align}
		x+v t &= \tilde x_1 A + v\tilde t A\\
		x+v s &= \tilde x_2 A + v\tilde s A
	\end{align}
			with $\tilde t, \tilde s \in \M[1,d]$, $\tilde x_1, \tilde x_2 \in \M[n,d]$, $\tilde x_1 + v \tilde t, \tilde x_2 + v \tilde s \in K$ and $v^t  \tilde x_1 = v^t  \tilde x_2 = 0$.
	But since $v^t \tilde x_1 A = v^t \tilde x_2 A = 0$, by uniqueness of the decomposition $\M = v^{m} \oplus v^{\perp m}$, we must have
	$t=\tilde t A, s=\tilde s A$ and $\tilde x_1 A = x = \tilde x_2 A$.
	Now using that $A$ has rank $d$, from $\tilde x_1 A = \tilde x_2 A$ we get $\tilde x_1 = \tilde x_2$.

	Finally, we obtain
	\[z = (\tilde x_1 +v \frac{\tilde t - \tilde s}2) A\]
	with $\tilde x_1 + v \tilde t, \tilde x_1 + v \tilde s \in K, v^t  \tilde x_1 = 0$ and we conclude that $z \in (\bar S_v K) A$.

\medskip\ref{prp_sym_product}
	By formula \eqref{eq_sym_def}, $ \bar S_v (K_1 \times \cdots \times K_r)$ is formed by the points 
	\[(x_1, \ldots, x_r) + v \frac{(s_1, \ldots, s_r)-(t_1, \ldots, t_r)}2 = (x_1 + v \frac{t_1-s_1}2, \cdots, x_r + v \frac{t_r-s_r}2 )\]
	where $x_i + v s_i, x_i + v t_i \in K_i$. Then
			\begin{align}
				\bar S_v (K_1 \times \cdots \times K_r)
				&= \left\{(x_i + v \frac{t_i-s_i}2)_i: x_i + v s_i, x_i + v t_i \in K_i \right\} \\
				&= \bar S_v K_1 \times \cdots \times \bar S_v K_r.
			\end{align}

\medskip\ref{prp_sym_right_submersion}
	Consider the $d \times m$ matrix
	\[T_{d,m} = \left( \begin{array}{c} \I_m \\ 0_{d-m,m}\end{array} \right)\]
	so that $\overline{T_{d,m}}(x_1, \ldots, x_d) = (x_1, \ldots, x_m)$ for $x_i \in \R^n$.
	The proposition is true replacing $A$ by $T_{d,m}$, by property \ref{prp_sym_product} with $r=2, m_1 = m, m_2 = m-d, K_1 = K, K_2 = \R^{d-m}$.

	If $d=m$, the matrix $A$ must be invertible and the proposition is also true by property \ref{prp_sym_right_commutation}.

	For general $A$, we use that $A$ can be decomposed as $A = \varphi T_{d,m} \psi$ where $\varphi \in GL_m, \psi \in GL_d$.
	We conclude by equation \eqref{def_functoriality} and properties \ref{prp_sym_product} and \ref{prp_sym_right_commutation} again.
\end{proof}

The following properties of the symmetrization on functions are obtained immediately.

\begin{proposition}
	\label{prop_func_properties}
\ 
	\begin{enumerate}[ label=\ref{prop_func_properties}-\arabic* ]
\item
	\label{prp_steiner_concave_is_invariant}
	 An $m$-quasi-concave function $f:\M \to \R$ is Steiner-concave if, and only if, $f^{(v)} = f$ for every $v \in S^{n-1}$.

\item
	\label{prp_func_partial}
	If $g:\R^n \to \R$ is non-negative and quasi-concave, then the function $f:\M \to \R$ defined by $f(x_1, \cdots, x_m) = g(x_i)$ is $m$-quasi-concave, and
	\[f^{(v)}(x_1, \cdots, x_m) = g^{(v)}(x_i).\]
\item
	\label{prp_func_onedimensional}
	For the case $n=1$, recall that $S^0 = \{-1,+1\}$.
	If $K \subseteq \R^m = \M[1,m]$, 
	\[f^{(\pm 1)}(x) = \int_0^\infty \chi_{\frac 12D \{f \geq t\}}(x) d t.\]
\item
	\label{prp_func_right_commutation}
	If $A \in \M[d,m]$ is a rank $m$ matrix (we implicitly assume $d \geq m$), then
	\[(f \circ \bar A)^{(v)} = f^{(v)} \circ \bar A.\]
\end{enumerate}
\end{proposition}

The measure-preserving property of the usual Steiner symmetrization is a useful property that is lost for the operator $\bar S_v$.
Nevertheless, $\bar S_v$ still preserves some weaker measures on convex bodies like the mean width and the volumes of some projections.

\begin{proposition}
	\label{prop_meanwidth}
	Let $K \subseteq \M$ be a convex body with the origin in the interior.
	For any $A \in GL_m$, $v \in S^{n-1}$ and $p \geq 1$,
	\[ \int_{\s} h_{(\bar S_v K) A}(z)^p d z \leq \int_{\s} h_{K A}(z)^p d z. \]
	If $p>1$ there is equality if and only if $R_v K = K$.
\end{proposition}
\begin{proof}
	By property \ref{prp_sym_right_commutation} it suffices to consider the case $A = I_m$.

	Recall that $R_v \in \M[n,n]$ is the matrix of the reflection in $\R^n$ with respect to $v^\perp$.

	Let $x \in v^{\perp m}$, $x + v t, x + v s \in K$ and $z \in \M$. Then
	\[\langle x + v \frac{t-s}2, z\rangle = \frac 12(\langle x+vt,z \rangle + \langle x-vs,z \rangle) \leq \frac 12( h_K(z) + h_{K}(R_v z) )\]
	Taking supremum for all $x + v t, x + v s \in K$ and using \eqref{eq_sym_def}, we get $h_{\bar S_v K}(z) \leq \frac 12(h_K(z) + h_{R_v K}(z))$.
	Since $R_v \S = \S$, we obtain
	\[\|h_{\bar S_v K}\|_p \leq \frac 12( \|h_{K}\|_p+\|h_{R_v K}\|_p ) = \|h_{K}\|_p\]
	which proves the inequality.

	For the equality case, notice that if $\|h_{\bar S_v K}\|_p =\|h_{K}\|_p$, then
	\[\frac 12 \|h_{K} + h_{R_v K}\|_p \geq \|h_{\bar S_v K}\|_p =\|h_{K}\|_p = \frac 12( \|h_{K}\|_p+\|h_{R_v K}\|_p ),\]
	which for $p>1$ implies that that there exists $\lambda \in \R$ with $h_K(z) = \lambda h_{R_v K}(z)$ for almost every $z$.
	Integrating on $z \in \S$ on both sides we get $w(K) = \lambda w(R_v K)$ which implies that $\lambda=1$.
	
	Thus we obtain $h_K(z) = h_{R_v K}(z)$ for almost every $z \in \S$. By continuity, this occurs for every $z \in \S$ and we get $K = R_v K$.
\end{proof}

For $e \in \M[m,1] \setminus \{0\}$ the function $\bar e:\M \to \R^n$ is a linear projection onto an $n$-dimensional space.
In particular, if $e_i$ is the $i$-th canonical vector, $\overline{e_i} = \pi_i:\M \to \R^n$ is the projection onto the $i$-th column, this is, $\pi_i(x)$ is the $i$-th column of $x$.
\begin{proposition}
	For a convex body $K \subseteq \M$ and $e \in \M[m,1]$,
	\[
		\vol{\bar e (\bar S_v K)} \leq \vol{\bar e(K)}.
	\]
	In particular,
	\[
		\vol{\pi_i(\bar S_v K)} \leq \vol{\pi_i(K)}.
	\]
\end{proposition}
\begin{proof}
	By the properties \ref{prp_sym_right_commutation} and \ref{prp_sym_is_steiner}, we have
	\[\vol{\bar e(\bar S_v K)} =\vol{(\bar S_v K) e} \leq \vol{S_v (K e)} = \vol{K e} = \vol{\bar e(K)}\]
\end{proof}

\section{Rogers--Brascamp--Lieb--Luttinger Inequalities}
\label{sec_inequalities}
We start with an inequality of convex sets which is the geometric core of Theorem \ref{thm_bll_with_linears}.

\begin{theorem}
	\label{thm_intersets}
	Let $k \geq 1$ and $K_i \subseteq \M[1,m] = \R^m, 1 \leq i \leq k$, be convex sets, (not necessarily convex bodies, possibly with infinite volume), then
	\[\vol[m]{\bigcap_{i=1}^k K_i} \leq \vol[m]{\bigcap_{i=1}^k \frac 12D K_i}.\]
	If the left-hand side has infinite volume, then so has the right-hand side.

	Moreover, if $K,L \subseteq \M[1,m] = \R^m$ are convex sets, and $K$ is origin-symmetric and has finite volume,
	\[\vol[m]{K \setminus L} \geq \vol[m]{K \setminus \frac 12D L}.\]
\end{theorem}
\begin{proof}
	Let $x \in \frac 12D(\bigcap_{i=1}^k K_i)$, then $x = \frac {a-b}2$ with $a,b \in \bigcap_{i=1}^k K_i$.
	In particular, $x \in \frac 12D K_i$ for every $i$. This proves that
	\[\frac 12D \left(\bigcap_{i=1}^k K_i\right) \subseteq \bigcap_{i=1}^k\frac 12D K_i.\]
	By the Brunn-Minkowski inequality,
	\[\vol[m]{\bigcap_{i=1}^k K_i} \leq \vol[m]{\frac 12D \bigcap_{i=1}^k K_i} \leq \vol[m]{\bigcap_{i=1}^k\frac 12D K_i}\]
	and the first part of the theorem follows.

	For the second part use that $K = \frac 12 DK$ to obtain
	\begin{align}
		\vol[m]{K \setminus L}
		&= \vol[m] K - \vol[m]{K \cap L} \\
		&\geq \vol[m] K - \vol[m]{\frac 12D K \cap \frac 12D L} \\
		&= \vol[m] K - \vol[m]{K \cap \frac 12D L} \\
		&= \vol[m]{K \setminus \frac 12D L}. \\
	\end{align}
\end{proof}

We remark here that Theorem \ref{thm_intersets} is neither stronger nor weaker than the inequality $\vol[m]{\bigcap_{i=1}^k K_i} \leq \max_{i} \{\vol[m]{K_i}\}$ which is obtained applying inequality \eqref{eq_bll_original} with the symmetric decreasing rearrangement.
For example, if $K_1, K_2$ are two origin-symmetric convex bodies of volume $1$ with very small intersection, then Theorem \ref{thm_intersets} gives an equality, while the inequality $\vol[m]{K_1 \cap K_2} < 1$ is strict.

\begin{theorem}
	\label{thm_symsets}
	Let $k \geq 1$ and $K_i \subseteq \M, 1 \leq i \leq k$ be convex sets (not necessarily convex bodies), and $v \in S^{n-1}$, then
	\[\Vol{\bigcap_{i=1}^k K_i} \leq \Vol{\bigcap_{i=1}^k \bar S_v K_i}.\]
	If the left-hand side has infinite volume, then so has the right-hand side.
	Moreover, if $K$ satisfies $\bar S_v K = K$ and has finite volume,
	\[\Vol{K \setminus L} \geq \Vol{K \setminus \bar S_v L}.\]
\end{theorem}
\begin{proof}
	By Fubini's theorem and Theorem \ref{thm_intersets},
	\begin{align}
		\Vol{\bigcap_{i=1}^k K_i}
		&= \int_{v^{\perp m}} \vol[m]{(x+v^m) \cap \bigcap_{i=1}^k K_i} d x \\
		&= \int_{v^{\perp m}} \vol[m]{\bigcap_{i=1}^k \left( (x+v^m) \cap K_i \right)} d x \\
		&\leq \int_{v^{\perp m}} \vol[m]{\bigcap_{i=1}^k \frac 12 D \left( (x+v^m) \cap K_i \right)} d x \\
		&= \int_{v^{\perp m}} \vol[m]{\bigcap_{i=1}^k (x+v^m) \cap \bar S_v K_i} d x \\
		&= \int_{v^{\perp m}} \vol[m]{(x+v^m) \cap \bigcap_{i=1}^k \bar S_v K_i} d x 
		= \Vol{\bigcap_{i=1}^k \bar S_v K_i}
	\end{align}
	The second statement follows as in the second part of the proof of Theorem \ref{thm_intersets}.
\end{proof}

\begin{theorem}
	\label{thm_bll_direct}
	Let $k \geq 1$ and $f_i:\M \to \R, 1 \leq i \leq k$, be non-negative and $m$-quasi-concave.
	Then 
	\[\int_{\M} \prod_{i=1}^k f_i(x) d x \leq \int_{\M} \prod_{i=1}^k f_i^{(v)}(x) d x\]
\end{theorem}
\begin{proof}
	By the layer cake formula and Theorem \ref{thm_intersets},
	\begin{align}
		\int_{\M} \prod_{i=1}^k f_i(x) d x
		&=\int_{\M} \prod_{i=1}^k \int_0^\infty \chi_{\{f_i \geq t_i\}}(x) d t_i d x \\
		&=\int_0^\infty \cdots \int_0^\infty \int_{\M} \prod_{i=1}^k \chi_{\{f_i \geq t_i\}}(x) d x d t_i\\
		&= \int_0^\infty \cdots \int_0^\infty \Vol{\bigcap_{i=1}^k \{f_i \geq t_i\}} d t_1 \cdots d t_k \\
		&\leq \int_0^\infty \cdots \int_0^\infty \Vol{\bigcap_{i=1}^k \bar S_v \{f_i \geq t_i\}} d t_1 \cdots d t_k \\
		&= \int_0^\infty \cdots \int_0^\infty \Vol{\bigcap_{i=1}^k \{f_i^{(v)} \geq t_i\}} d t_1 \cdots d t_k \\
		&= \int_{\M} \prod_{i=1}^k f_i^{(v)}(x) d x
	\end{align}
\end{proof}

Finally, we can deduce the Rogers--Brascamp--Lieb--Luttinger inequality in its full generality.

\begin{proof}[Proof of Theorem \ref{thm_bll_with_linears}]
	Apply Theorem \ref{thm_bll_direct} to the functions $f_i \circ \overline{L_i}$, and use property \ref{prp_func_right_commutation}.
\end{proof}

For convex functions we obtain the following:

\begin{theorem}
	\label{thm_bll_convex}
	Let $f:\M \to \R$ be a non-negative and Steiner-concave function, and let $g:\M \to \R$ be non-negative and $m$-quasi-convex, then
	\[\int_{\M} f(x) g(x) d x \geq \int_{\M} f(x) g_{(v)}(x) d x \]
\end{theorem}
\begin{proof}
	By the layer cake formula, Definition \ref{def_funcsym_convex}, and the second part of Theorem \ref{thm_symsets},
	\begin{align}
		\int_{\M} f(x) g(x) d x
		&= \int_0^\infty \int_0^\infty \Vol{\{f \geq t\}\cap \{g \geq t_2\}} d t_1 d t_2 \\
		&= \int_0^\infty \int_0^\infty \Vol{\{f \geq t\} \setminus \{g < t_2\}} d t_1 d t_2 \\
		&\ge \int_0^\infty \int_0^\infty \Vol{\bar S_v \{f \geq t\} \setminus \bar S_v \{g < t_2\}} d t_1 d t_2 \\
		&= \int_0^\infty \int_0^\infty \Vol{\{f^{(v)} \geq t\} \setminus \{g^{(v)} < t_2\}} d t_1 d t_2 \\
		&= \int_0^\infty \int_0^\infty \Vol{\{f^{(v)} \geq t\} \cap \{g^{(v)} \geq t_2\}} d t_1 d t_2 \\
		&= \int_{\M} f(x) g_{(v)}(x) d x
	\end{align}
	where we used that $f^{(v)} = f$, since $f$ is Steiner-concave.
\end{proof}

\section{Some Applications}
\label{sec_applications}

\subsection{Particular cases of Theorem \ref{thm_bll_with_linears}}

As mentioned in the introduction, we see that Theorem \ref{thm_bll_paouris_pivovarov} is a particular case of Theorem \ref{thm_bll_with_linears}.

A particular case of Theorem \ref{thm_bll_direct} (which is Theorem \ref{thm_bll_with_linears} with $m=d$ and $L_i = \I$) is the following
\begin{corollary}
	\label{cor_general_pp_noGCC}
	Let $f:\M \to \R$ be Steiner-concave, and $K \in \conbod$, then
	\[\int_K f(x) d x \leq \int_{\bar S_v K} f(x) d x.\]
\end{corollary}

For example, it was proven in \cite[Corollary 4.2]{PP17-2} that $(x_1, \ldots, x_m) \mapsto \vol{\co(\{x_1, \ldots, x_m\})}$ is Steiner-convex, where $\co$ denotes the convex hull.
We deduce immediately,
\begin{corollary}
	\label{cor_Groemer_noGCC}
	Let $K \in \conbod$ and $\gamma:(0,\infty) \to (0,\infty)$ be a decreasing function, then
	\[\int_K \gamma(\vol{\co(\{x_1, \ldots, x_m\})}) d x \leq \int_{\bar S_v K} \gamma(\vol{\co(\{x_1, \ldots, x_m\})}) d x.\]
\end{corollary}

For the case $n=1$, Theorem \ref{thm_bll_direct} gives
\begin{corollary}
	Let $k \geq 1$ and $f_i:\R^m \to \R, 1 \leq i \leq k$, be quasi-concave functions, then
	\label{cor_bll_onedimensional}
	\[
		\int_{\R^m} \prod_{i=1}^k f_i(x) d x \leq \int_{\R^m} \prod_{i=1}^k f_i^{(1)}(x) d x.
	\]
	where $f_i^{(1)}$ is defined as in property \eqref{prp_func_onedimensional}.
\end{corollary}

Notice that, as in the remark after Theorem \ref{thm_intersets}, this inequality cannot be proven directly with inequality \eqref{eq_bll_original}.

\subsection{Operator Norms}
\label{subsec_operatornorms}

For any convex bodies $K \subseteq \R^m, L \subseteq \R^n$ consider the set
\[B_{K,L} = \{x \in \M: x w \in L,\ \ \forall w \in K\}\]
where we identify vectors in $\R^m$ and $\R^n$ with column vectors in $\M[m,1]$ and $\M[n,1]$ respectively.

If $K,L$ are origin-symmetric, the set $B_{K,L} \subseteq \M$ is the unit ball in the operator norm of the set of linear maps between the Banach spaces
$(\R^m, \|\cdot\|_K) \to (\R^n, \|\cdot\|_L)$.
Each of these maps induces a linear map of the dual spaces $(\R^m, \|\cdot\|_{L^\circ}) \to (\R^n, \|\cdot\|_{K^\circ})$
by transposition. In general, for convex bodies $K,L$ containing the origin in the interior, this reads as 
\begin{equation}
	\label{eq_transposition}
	B_{K,L}^t = B_{L^\circ, K^\circ}.
\end{equation}

In \cite{HLPRY23_2} the following result is obtained as a limit case of the higher-order $L^p$ Petty-projection inequality.
We give here a shorter and clearer proof.
\begin{theorem}
	\label{thm_sym_operatornorm}
	For any convex bodies $K \subseteq \R^m, L \subseteq \R^n$,
	\[\bar S_v B_{K,L} \subseteq B_{K,S_v L}\]
\end{theorem}
\begin{proof}
	Take any $e \in K \subseteq \M[m,1]$, so that $(B_{K,L}) e \subseteq L$. By \eqref{prp_sym_right_commutation} and \eqref{prp_sym_monotonicity},
	\[(\bar S_v B_{K,L}) e \subseteq \bar S_v (B_{K,L} e) \subseteq \bar S_v L.\]
	Since $e \in K$ was arbitrary, this implies that $\bar S_v B_{K,L} \subseteq B_{K,S_v L}$.
\end{proof}

Theorem \ref{thm_sym_operatornorm} then implies that 
\[\Vol{B_{K,L}} \leq \Vol{B_{K,B_L}}\]
where $B_L$ is the centered Euclidean ball of the same volume as $L$.
As in the proof of \cite[Theorem~1.4]{HLPRY23_2}, the transposition argument mentioned above together with the Blaschke-Santaló inequality yields also
\[\Vol{B_{K,L}} \leq \Vol{B_{B_K,L}}\]

Some more can be said:
\begin{theorem}
Let $F:\M \to \R$ be non-negative and Steiner-concave, and let $K \subseteq \R^m, L \subseteq \R^n$ be convex bodies, then
	\begin{equation}
		\label{eq_integration_right}
		\int_{B_{K,L}} F(x) d x \leq \int_{B_{K,B_L}} F(x) d x.
	\end{equation}
	Let $F:\M[m,n] \to \R$ be non-negative and Steiner-concave, and let $K \subseteq \R^m, L \subseteq \R^n$ be convex bodies containing the origin in their interior, then
	\begin{equation}
		\label{eq_integration_left}
	\int_{B_{K,L}} F(x^t) d x \leq \int_{B_{B_{K^\circ},L}} F(x^t) d x.
	\end{equation}

	In particular, for $n=m$ and $\beta \in (-1,0)$
	\begin{align}
		\int_{B_{K,L} \cap GL_n } |\det(x)|^\beta d x 
		& \leq \left( \frac{\vol{L}\vol{K^\circ}}{\vol{\B}^2}\right)^{n+\beta} \int_{B_{\B,\B} \cap GL_n } |\det(x)|^\beta d x \label{eq_integration_both}.
	\end{align}
	If additionally, $K$ is origin-symmetric,
	\begin{align}
		\int_{B_{K,L} \cap GL_n } |\det(x)|^\beta d x 
		& \leq \left( \frac{\vol{L}}{\vol{K}}\right)^{n+\beta} \int_{B_{\B,\B} \cap GL_n } |\det(x)|^\beta d x \label{eq_integration_BS} \\
		&= \int_{B_{B_K,B_L} \cap GL_n } |\det(x)|^\beta d x \label{eq_integration_rescaling}.
	\end{align}

\end{theorem}
\begin{proof}
	The first inequality \eqref{eq_integration_right} is a direct consequence of Theorem \ref{thm_bll_direct}.
	The second one, \eqref{eq_integration_left} comes from \eqref{eq_integration_right}, formula \eqref{eq_transposition} and the change of variables $x \mapsto x^t$. 

	For the third inequality \eqref{eq_integration_both}, first observe that the function $x \in \M[n,n] \mapsto |\det(x)|^\beta \in [0, \infty]$ is locally integrable for $\beta \in (-1,0)$.

	Take any continuous non-negative and strictly decreasing function $\gamma:[0,\infty) \to \R$.
	Since $x \mapsto \gamma(|\det(x)|)$ is Steiner-concave and invariant by transposition, we may apply \eqref{eq_integration_right}, the change of variables $x \mapsto x^t$  and \eqref{eq_integration_left}, to obtain
	\begin{align}
		\int_{B_{K,L}} \gamma(|\det(x)|) d x
		&\leq \int_{B_{K,B_L}} \gamma(|\det(x)|) d x \\
		&\leq \int_{B_{B_L^\circ, K^\circ}} \gamma(|\det(y)|) d y \\
		&\leq \int_{B_{B_L^\circ, B_{K^\circ}}} \gamma(|\det(y)|) d y \\
		&\leq \int_{B_{(B_{K^\circ})^\circ,B_L}} \gamma(|\det(x)|) d x.
	\end{align}
	Now we may approximate $t \in [0,\infty) \mapsto t^\beta \in [0,\infty]$ by continuous non-increasing functions, and by standard convergence theorems, we get
	\begin{align}
		\int_{B_{K,L} \cap GL_n } |\det(x)|^\beta d x 
		&\leq \int_{B_{(B_{K^\circ})^\circ,B_L} \cap GL_n} |\det(x)|^\beta d x \\
		&= \left(\frac{\vol{L}\vol{K^\circ}}{\vol{\B}^2}\right)^{n+\beta}\int_{B_{\B,\B} \cap GL_n} |\det(y)|^\beta d y,
	\end{align}
	where we used the change of variables $x = \left(\frac{\vol{L}\vol{K^\circ}}{\vol{\B}^2}\right)^{1/n} y$.
	Finally, inequality \eqref{eq_integration_BS} follows from \eqref{eq_integration_both} and the Blaschke-Santaló inequality, and equality \eqref{eq_integration_rescaling} follows from \eqref{eq_integration_BS} and the change of variables $y = (\vol{K}/\vol{L})^{1/n} x$.
\end{proof}

\subsection{Schneider's Difference Body}
For $K \subseteq \R^n$ and $m=1$, Schneider \cite{Sch70} defined the higher-order difference body as
\[D^m K = \left\{ (x_1, \ldots, x_m) : K \cap (K+x_1) \cap \cdots \cap (K+x_m) \neq \emptyset \right\}. \]

Here we prove that Schneider's operator intertwines with the fiber symmetrization.

\begin{theorem}
	\label{thm_schneider_difference_body}
	If $K \subseteq \R^n$ is a convex body, then $\bar S_v(D^m K) \supseteq D^m (S_v K)$.
\end{theorem}
\begin{proof}
	Define the matrix $P_m \in \M[m+1,m]$ by
	\[
		P_m = \left(
	\begin{array}{cccc}	1 & 1 &\cdots & 1 \\
				-1& 0 &\cdots & 0 \\
				0& -1 &\cdots & 0 \\
				\cdots & \cdots & \cdots & \cdots \\
				0& 0 &\cdots & -1 \\
	\end{array}
	\right).
	\]
	It is easy to see that $D^m K = K^{m+1} P_m$.
	By properties \eqref{prp_sym_right_commutation} and \eqref{prp_sym_product},
	\[\bar S_v(D^m K) \supseteq \bar S_v(K^{m+1}) P_m = (S_v K)^{m+1} P_m = D^m(S_v K).\]
\end{proof}
In \cite{Sch70} Schneider conjectured that the volume of $D^m K$ is minimized among convex sets if $K$ is an ellipsoid.
Unfortunately Theorem \ref{thm_schneider_difference_body} is not appropriate to study this conjecture, since the operator $\bar S_v$ is volume-increasing.
However, combining Theorem \ref{thm_schneider_difference_body} and Proposition \ref{prop_meanwidth} we get a (much weaker) result in this direction.

\begin{theorem}
	Among all convex bodies $K \subseteq \R^n$ of a fixed volume, Euclidean balls minimize the mean width of $D^m K$.
\end{theorem}

Some more can be said:
\begin{theorem}
	Let $L \in \conbod$ satisfy $R_v L = L$ for every $v \in S^{n-1}$ (see Section \ref{sec_invariant_sets} for examples of invariant sets).
	Among all convex bodies $K \subseteq \R^n$ of a fixed volume, Euclidean balls minimize the integral
	\[\int_L h_{D^m K}(x) d x.\] 
\end{theorem}
\begin{proof}
	As in the proof of Proposition \ref{prop_meanwidth}, for every $E \in \conbod$, 
	\[\int_L h_{\bar S_v E}(x) d x \leq \int_L h_E(x) d x.\]
	For $E = D^m K$, by Theorem \ref{thm_schneider_difference_body},
	\[
		\int_L h_{D^m S_v K} \leq\int_L h_{\bar S_v D^m K} \leq\int_L h_{D^m K} .
	\]
\end{proof}

\section{Invariant Sets}
\label{sec_invariant_sets}

In this section we study the sets $K$ that remain invariant under $\bar S_v$.

\begin{definition}
	A convex body $K \in \conbod$ is said to be $O_n$ invariant if $R K = K$ for every $n$-dimensional rotation $R \in O_n$.
\end{definition}
Notice that this does not necessarily imply that $K$ is a ball, since $R$ is an $n\times n$ matrix, not $nm \times nm$.

Contrary to the classical case $m=1$ where we only have the centered Euclidean balls of different radius, many convex bodies in $\M$ are $O_n$ invariant. 
As examples we have the unit ball $\B[nm]$, the product of balls $(\B)^m$, the unit ball of any operator norm of the type $B_{K,\B}$ and the ball of any unitary invariant norm (in particular any Schatten class). Additional examples are the convex bodies $\Pi_{\B,p} K$ defined in \cite{HLPRY23_2}.

Here we show that $O_n$ invariant convex bodies characterize the sets invariant by the operator $\bar S_v$.
\begin{proposition}
	\label{prop_invariance_equivalence}
	Let $K \in \conbod$. The following statements are equivalent:

	\begin{enumerate}
		\item For every $v \in \s$, $\bar S_v K = K$.
		\item For every $v \in \s$, $R_v K = K$.
		\item $K$ is $O_n$ invariant.
	\end{enumerate}
\end{proposition}
If $K$ is also origin-symmetric, it is the unit ball of a matrix norm $\|\cdot\|_K$.
Then $O_n$ invariance just means left-orthogonal invariance of the norm.

\begin{proof}
	Since the reflections $R_v, v\in \s$ generate $O_n$, $(2) \Leftrightarrow (3)$ are clearly equivalent.

	To see $(1) \Leftrightarrow (2)$, for any $v \in \s, y \in v^{\perp m}$ consider the set
	\[K_{v,y} = \{t \in \M[1,m]: y+vt \in K\}.\]
	The property $\bar S_v K = K$ holds if and only if for any $y \in v^{\perp m}$,
	\[K_{v,y} = (\bar S_v K)_{v,y} = (\frac 12DK)_{v,y},\]
	and this is equivalent to $K_{v,y} = -(K_{v,y}) = (R_v K)_{v,y}$, which is equivalent to $(2)$.
\end{proof}

For functions we obtain the following characterization:
\begin{proposition}
	An $m$-quasi-concave function $f$ is Steiner-concave if and only if $f(R x)=f(x)$ for every $R \in O_n$.
\end{proposition}
\begin{proof}
	By definition, an $m$-quasi-concave function $f$ is Steiner-concave if and only if $f(R_v x) = f(x)$.
	Since a function is uniquely determined by its level sets, an $m$-quasi-concave function $f$ is Steiner-concave if and only if the level sets are $O_n$ invariant.
	The result follows by Proposition \ref{prop_invariance_equivalence}.
\end{proof}

Finally, as in the classical case, every convex body can be transformed into an extremal set.

\begin{theorem}
	Let $K \subseteq \M$ be a convex body, then there exists an $O_n$ invariant convex body $L \subseteq \M$ such that for every $\varepsilon > 0$ there are $k \geq 1$ vectors $v_1, \ldots, v_k$ with
	\[d_H(\bar S_{v_k} \circ \cdots \circ \bar S_{v_1} K, L) < \varepsilon,\]
	where $d_H$ is the Hausdorff distance.
\end{theorem}
\begin{proof}
	For a $k$-tuple of vectors $\bar v = (v_1, \ldots, v_k)$ we denote by $\bar S_{\bar v} = \bar S_{v_k} \circ \cdots \circ \bar S_{v_1}$ the composition of the symmetrizations $\bar S_{v_i}$.

	Let $R>r>0$ be such that $B(0,R) \supseteq K \supseteq B(0,r)$.
	By property \eqref{prp_sym_monotonicity} and the fact that balls are $O_n$ invariant, we always have $B(0,R) \supseteq \bar S_{\bar v} K \supseteq B(0,r)$.
	Fix any $p>1$ and consider the infimum
	\[w_{p,\text{min}} = \inf\{\|h_{\bar S_{\bar v} K}\|_p: k\geq 1, \bar v = (v_1, \ldots, v_k), v_i \in S^{n-1}\}\]
	which in view of the inclusion $\bar S_{\bar v} K \supseteq B(0,r)$, has to be strictly positive.
	Take a sequence of tuples of vectors $\bar v^{(i)}$ such that $\|h_{\bar S_{\bar v^{(i)}} K}\|_p \to w_{p,\text{min}}$.
	Since $B(0,R) \supseteq \bar S_{\bar v^{(i)}} K$ we may apply the Blaschke selection theorem to obtain a subsequence (again denoted by $\bar v^{(i)}$)
	such that $S_{\bar v^{(i)}} K$ converges in the Hausdorff distance to a convex body $L$.
	By continuity, $\|h_L\|_p = w_{p,\text{min}}$.

	Take any $v \in S^{n-1}$. By continuity and Proposition \ref{prop_meanwidth},
	\begin{align}
		w_{p,\text{min}}
		&= \|h_{L}\|_p \\
		&\geq \|h_{\bar S_v L}\|_p \\
		&= \lim_{i \to \infty} \|h_{\bar S_v \bar S_{\bar v^{(i)}}} K \|_p \\
		&\geq w_{p,\text{min}}
	\end{align}
	and we get $\|h_{L}\|_p =\|h_{\bar S_v L}\|_p$.
	The equality case of Proposition \ref{prop_meanwidth} (notice that we chose $p>1$) implies $L = R_v L$.

	Since $v \in S^{n-1}$ was arbitrary, $L$ must be $O_n$ invariant by Proposition \ref{prop_invariance_equivalence}.
\end{proof}

\subsection*{Concluding Remarks}
The results in this last section point to the fact that repeated application of $\bar S_v$ makes a matrix norm $\|\cdot\|_K$ more and more left-invariant. Of course one can also consider the transformation $K \mapsto (\bar S_v K^t)^t$ which makes the norm more right-invariant.
If the norm is transformed by a sequence of left and right symmetrizations into a limiting norm $\|\cdot\|_L$, then it will be unitarily invariant and, by singular value decomposition, $L$ will be uniquely determined by the set of singular values of the matrices inside $L$.
Unfortunately it appears that the set of singular values may vary from $K$ to $\bar S_v K$ for arbitrary $K$, so it is not clear how (if possible) to determine $L$ from $K$.

\subsection*{Acknowledgments}
The author was supported by Grant RYC2021-031572-I, funded by the Ministry of Science and Innovation/State Research Agency/10.13039 /501100011033 and by the E.U. Next Generation EU/Recovery, Transformation and Resilience Plan.

The author thank the reviewers for their careful reading and very helpful remarks.

\bibliographystyle{abbrv}
\bibliography{../references}
\end{document}